\documentclass[11pt]{amsart}
\usepackage{amsmath}
\usepackage{amsfonts}
\usepackage{amsthm}
\usepackage{amssymb}
\usepackage{mathrsfs}
\usepackage{stmaryrd}
\usepackage{a4wide}
\usepackage{hyperref}
\usepackage{enumitem}
\usepackage{faktor}

\makeatletter
\DeclareRobustCommand*{\mfaktor}[3][]
{
   { \mathpalette{\mfaktor@impl@}{{#1}{#2}{#3}} }
}
\newcommand*{\mfaktor@impl@}[2]{\mfaktor@impl#1#2}
\newcommand*{\mfaktor@impl}[4]{
   \settoheight{\faktor@zaehlerhoehe}{\ensuremath{#1#2{#3}}}%
   \settoheight{\faktor@nennerhoehe}{\ensuremath{#1#2{#4}}}%
      \raisebox{-0.5\faktor@zaehlerhoehe}{\ensuremath{#1#2{#3}}}%
      \mkern-4mu\diagdown\mkern-5mu%
      \raisebox{0.5\faktor@nennerhoehe}{\ensuremath{#1#2{#4}}}%
}
\makeatother

{\newtheoremstyle{case}{}{}{}{}{}{:}{ }{}
\theoremstyle{case}
}

\numberwithin{equation}{section}
\usepackage{enumitem}

\def\Z{\mathbf{Z}}

\def\Q{\mathbf{Q}}

\def\C{\mathbf{C}}

\def\Gal{\mathrm{Gal}}
\def\Qbar{\overline{\Q}}
\def\R{\mathbf{R}}

\DeclareMathOperator\Spec{Spec}

\newtheorem{theorem}{Theorem}[section]

\newtheorem{lemma}[theorem]{Lemma}
\newtheorem{remark}[theorem]{Remark}

\newtheorem{conj}[theorem]{Conjecture}

\usepackage{enumitem}

\def\Aut{\mathrm{Aut}}

\address{Joel Specter, Department of Mathematics, Johns Hopkins University, 3400 N. Charles Street, Baltimore, MD 21218, United States}
\email{jspecter@jhu.edu}

\subjclass[2010]{11R32, 37P15}

\begin{document}

\title[Surjective Arboreal Galois Representations]{Polynomials with Surjective Arboreal Galois Representations Exist in Every Degree}
\author{Joel Specter}
\begin{abstract}
Let~$E$ be a Hilbertian field of {{characteristic~$0$}}. R.W.K. Odoni conjectured that for every positive integer~$n$ there exists a polynomial~$f\in E[X]$ of degree~$n$ such that each iterate~$f^{\circ{k}}$ of~$f$ is irreducible and the Galois group of the splitting field of~$f^{\circ k}$ is isomorphic to the automorphism group of a regular,~$n$-branching tree of height~$k.$ We prove this conjecture when~$E$ is a number field.
\end{abstract}
\maketitle 

\section{Introduction}
\label{intro}
Given a polynomial~$f \in \Q[X],$ the roots of~$f$ are the most evident set on which the absolute Galois group acts. This note concerns the Galois action on the second most evident set: the set of roots of all compositional iterates of~$f.$ 

We begin by establishing some notation. All fields considered in this note have characteristic~$0$. If~$F$ is a field and~$f\in F[X]$ is a polynomial, for each positive integer~$k,$ we denote the~$k$-th iterate of~$f$ under composition by~$f^{\circ k}.$ The set of all pre-images of~$0$ under the iterates of~$f$ is denoted
~$$T_f:= \coprod_{k=0}^{\infty} \{ r \in \overline{F}: f^{\circ k}(r) = 0\}.$$
To organize~$T_f,$ we give it the structure of a rooted tree: a zero~$r_k$ of~$f^{\circ k}$ is connected to a zero~$r_{k-1}$ of~$f^{\circ(k-1)}$ by an edge if~$f(r_k) = r_{k-1}.$ We call~$T_f$ the pre-image tree of~$0.$ The absolute Galois group~$G_F$  of~$F$ acts on~$T_f$ by tree automorphisms. The resulting map~$$\rho_f: G_{F} \rightarrow \Aut(T_f)$$ is called the arboreal Galois representation associated to~$f.$ We will say~$\rho_f$ is \emph{regular} if~$T_f$ is a regular, rooted tree of degree equal to the degree of~$f.$

Interest in arboreal Galois representations originates from the study of prime divisors appearing in the numerators of certain polynomially-defined recursive sequences. Explicitly, given a polynomial~$f \in \Q[X]$ and an element~$c_0 \in \Q,$ one wishes to understand the density of the set of primes~$$S_{f, c_0} := \{p: v_p(f^{\circ n}(c_0)) >0  \text{ for some value of } n \}$$ inside the set of all prime integers. An observation, first made by Odoni in \cite{odoni1985prime}, is that one may bound this density from above using Galois theory. Specifically, if one excludes the primes~$p$ for which~$c_0$ and~$f$ are not~$p$-integral, a prime~$p$ is contained in~$S_{f, c_0}$ if and only if~$c_0$ is a root of some iterate of~$f \mod p.$ By the Chebotarev Density Theorem, the proportion of primes~$p$ for which~${f^{\circ k} \mod p}$ has a root is determined by the image of~$\rho_f.$  As a general principle, if a polynomial has an arboreal Galois representation with \textit{large} image, then \textit{few} primes appear in~$S_{f,c_{0}}.$ For specific results, we refer the reader to \cite{odoni1985prime} or \cite{jones2008density}.

In~\cite{odoni1985Galois}, Odoni showed that for any field~$F$ of characteristic~$0,$ the arboreal Galois representation associated to the generic monic, degree~$n$ polynomial~$$f_{gen}(X) := X^{n}  + a_{n-1}X^{n-1} + \ldots + a_1X + a_0 \in F(a_{n-1}, \ldots, a_0)[X]$$ is regular and surjective.\footnote{Jamie Juul has shown that the arboreal Galois representation associated to the generic monic, degree~$n$ polynomial over a field~$F$ of any characteristic is regular and surjective under the assumption that the characteristic of~$F$ and the degree~$n$ do not both equal~$2$ \cite{juul2014iterates}.} When~$F$ is Hilbertian, for example when~$F = \Q$, one expects that \textit{most} monic, degree~$n$ polynomials behave like~$f_{gen}$. Indeed, this expectation holds true for any finite number of iterates: for each~$k > 0,$ the set of monic, degree~$n$ polynomials~$f$ such that the Galois group of~$f^{\circ k}$ over~$F$ is smaller than the Galois group of~$f^{\circ k}_{gen}$ over~$F(a_{n-1},\ldots, a_0)$ is thin. Alas, in general, the intersection of the complement of countably many thin sets may be empty; therefore, Odoni's theorem does not imply the existence of any specialization with surjective arboreal Galois representation. He conjectures that such specializations exist.

\begin{conj}[\cite{odoni1985Galois}, Conjecture 7.5]\label{odoni's Conjecture}  Let~$E$ be a Hilbertian field of {characteristic~$0$}. For each positive integer~$n,$ there exists a monic, degree~$n$ polynomial~$f \in E[X]$ such that every iterate of~$f$ is irreducible and the associated arboreal Galois representation~$$\rho_{f}: G_{E} \rightarrow \Aut(T_f)$$ is surjective. 
\end{conj} 

In this note, we prove Odoni's conjecture when~$E$ is a number field. More generally, we prove Conjecture~\ref{odoni's Conjecture} for extensions of~$\Q$ that are unramified outside of finitely many primes of~$\Z.$ 

\begin{theorem}\label{small} If~$E/\Q$ is an algebraic extension that is unramified outside finitely many primes, then for each positive integer~$n$ there exists a positive integer~$a<n$ and infinitely many~$A\in \Q$ such that the polynomial~$$f_{a,A}(X) := X^{a}(X-A)^{n-a} + A$$ and all of its iterates are irreducible over~$E$ and the arboreal~$G_E$-representation associated to~$f_{a,A}$ is surjective.     
\end{theorem} 

Our choice to consider the polynomial families in Theorem~\ref{small} was inspired by examples of surjective arboreal Galois representations over~$\Q$ constructed by Robert Odoni and Nicole Looper. In~\cite{odoni1985prime}, Odoni shows that the arboreal~$G_{\Q}$-representation associated to~${X(X-1) +1}$ is regular and surjective.   In~\cite{looper2016dynamical}, Looper proves Conjecture~\ref{odoni's Conjecture} for polynomials over~$\Q$ of prime degree by analyzing the arboreal Galois representations associated to certain integer specializations of the trinomial family~$X^n - ntX^{n-1} + nt = X^{n-1}(X-nt) + nt$. 

In addition to our note, there have been a series of recent, independent works concerning Odoni's conjecture. Borys Kadets \cite{2018arXiv180209074K} has proved Conjecture~\ref{odoni's Conjecture} when~$n$ is even and greater than 19, and~$E = \Q.$ Robert Benedetto and Jamie Juul \cite{BJ2018} have proved Conjecture~\ref{odoni's Conjecture} when~$E$ a number field, and~$n$ is even or~$\Q(\sqrt{n},\sqrt{n-2}) \not\subseteq E$.

The organization of this paper is as follows. Section~\ref{sec 2} provides a criterion with which to check if an arboreal Galois representation contains a congruence subgroup~$\Gamma(N)$. This criterion is that the image of the arboreal Galois representation contains, up to conjugation, some set of preferred elements~$$\{\sigma_0 \} \cup \{ \sigma_{k} : k>N\}\cup \{\sigma_{\infty,N}\}$$ which topologically generate a subgroup containing~$\Gamma(N).$ In Section~\ref{sec 3}, we show that for various explicit choices of~$A$ and~$a$ there are prime integers~$$\{p_0\}\cup \{ p_k : k > 0 \} \cup \{p_{\infty}\}$$ such that the image of the inertia group~$I_{p_k} \leqslant  G_{\Q_{p_k}}$ under~$\rho_{f_{a,A}}$ contains an element conjugate to~$\sigma_k$ if~$k<\infty,$ and conjugate to either~$\sigma_{{\infty},1}$ or~$\sigma_{{\infty},0}$ if~$k = \infty$. By choosing~$A$ well, one can force~$p_k$ to lie outside any fixed, finite set of primes; hence if~$E/\Q$ is unramified outside finitely many primes, then there is a choice of~$a$ and~$A$ such that the image of~$G_{E}$ under~$\rho_{f_{a,A}}$ contains~$\Gamma(1)$. Given such a polynomial, its arboreal Galois representation is surjective if and only if its splitting field is an~$S_n$-extension. In Section~\ref{sec 4}, we prove there are infinitely many values of~$A$ and~$a$ for which the representation~$\rho_{f_{a,A}}: G_ E \rightarrow  \Aut(T_{f_{a,A}})$ is surjective by means of a Hilbert Irreducibility argument. 

\section{Recognizing Surjective Representations}\label{sec 2}

Fix a field~$F$ of {characteristic~$0$} and let~$f \in F[X]$ be a polynomial. For every non-negative integer~$N,$ let~$$T_{f,N} := \coprod_{k=0}^{N} \{ r \in \overline{F}: f^{\circ k}(r) = 0\} \subseteq T_f$$ denote the full subtree of~$T_f$ whose vertices have at most height~$N.$  The subtree~$T_{f,N}$ is stable under the action of~$\Aut(T_f).$ Let~$\Gamma(N) \leqslant \Aut(T_f)$ be the vertex-wise stabilizer of~$T_{f,N}$ in~$\Aut(T_f).$ In this section, we describe a condition under which the image of~$\rho_f$ contains~$\Gamma(N).$ Since~$\Gamma(0)$ equals~$\Aut(T_f),$ the case when~$N=0$ is of primary interest. 

To state our criterion, we introduce some terminology. For each non-negative integer~$k,$ we denote the splitting field of~$f^{\circ k}$ over~$F$ by~$F_{k}.$ If~$k$ is negative, we define~$F_k := F.$ By a \emph{branch} of the tree~$T_f,$ we mean a sequence of vertices~$(r_i)_{i=0}^{\infty}$ such that~$r_0 = 0$ and~$f(r_i) = r_{i-1}$ for~$i>0.$ The group~$G_F$ acts on the branches of~$T_f.$ If~$X$ is some set of branches and~$\sigma \in G_F,$ we say that~$\sigma$ acts \emph{transitively} on~$X$ if the closed, pro-cyclic subgroup~$\overline{\langle \sigma \rangle} \subset G_F$ stabilizes~$X$ and acts transitively in the usual sense.   
\vspace{10pt}

The following is a sufficient condition for the image of a regular aboreal Galois representation to contain~$\Gamma(N).$ 

\begin{lemma}\label{Jordan's Criterion} Let~$N$ be a non-negative integer,~$f\in F[X]$ be a monic polynomial of degree~$n$, and~${a<n}$ be a positive integer such that either~$a=1,$ or~$a <n/2$ and~$n-a$ is prime. Assume that all iterates of~$f$ are separable. Furthermore, assume that:
\begin{enumerate} 
\item\label{JC 1}  there is an element{~$\sigma_0 \in G_F$} which acts transitively on the branches of~$T_f,$  
\item\label{JC 2} there is an element{~$\sigma_{{\infty},N} \in G_F$} and a regular,~$(n-a)$-branching subtree~$T \subseteq T_f$ such that{~$\sigma_{{\infty},N}$} acts transitively on the branches of~$T,$ and
\item\label{JC 3} for every positive integer~$k> N,$ there is an element~$\sigma_k \in \Gal(F_k/F_{k-1})$ which acts on the roots of~$f^{\circ k}$ in~$F_k$ as a transposition,
\end{enumerate}
\noindent then all iterates of~$f$ are irreducible, and the image of the arboreal Galois representation associated to~$f$ contains~$\Gamma(N).$  
\end{lemma} 

\begin{proof} Since all iterates of~$f$ are separable, Hypothesis \ref{JC 1} implies that all iterates of~$f$ are irreducible. We show that~$\Gamma(N)$ is contained in the image of~$\rho_f.$  

For all integers~$k>N,$ the subgroup~$\Gamma(k) \leqslant \Gamma(N)$ is finite index, and~$\Gamma(N)$ is isomorphic to the inverse limit~$\varprojlim_{k>N} \Gamma(N)/\Gamma(k).$ We regard~$\Gamma(N)$ as a topological group with respect to the topology induced by the system of neighborhoods~$\{\Gamma(k)\}_{k>N}$. The map~$\rho_f:G_F \rightarrow \Aut(T_f)$ is continuous in this topology. Since~$G_F$ is compact, the image,~$\rho_f(G_F),$ is closed.

To show that the closed subgroup~$\rho_f(G_F)$ contains~$\Gamma(N),$ it suffices to show that for all~{$k$ greater than~$N$}  \begin{equation}\label{eq 2.1} (\rho_f(G_F) \cap \Gamma(k-1))/ (\rho_f(G_F) \cap \Gamma(k)) =  \Gamma(k-1)/ \Gamma(k). \end{equation} Fix an integer~$k > N.$ Concretely,~$\Gamma(k-1)/\Gamma(k)$ is the group of permutations~$\sigma$ of the roots of~$f^{\circ k}$ which satisfy the relation~$f(\sigma(r_k)) = f(r_k).$ For each root~$\pi$ of~$f^{\circ(k-1)},$ let~$X_\pi$ denote the set of roots of~$f(X) - \pi$ in~$\overline{F}$. The group~$\Gamma(k-1)/\Gamma(k)$ stabilizes~$X_{\pi},$ and there is an isomorphism~\begin{equation}\label{eq 2.2}  \Gamma(k-1)/\Gamma(k) \cong \bigoplus_{\substack{ \pi \in \overline{F}  \\ f^{\circ(k-1)}(\pi) =0}}  S_{X_\pi} \end{equation} given by the direct sum of the restriction maps. Note that~$\Gal(F_{k}/F_{k-1})$  is the subquotient of~$G_{F}$ which is mapped isomorphically to~${(\rho_f(G_F) \cap \Gamma(k-1))/ (\rho_f(G_F) \cap \Gamma(k))}$ via the map induced by~$\rho_f.$

To show Equation~\eqref{eq 2.1} holds (and therefore prove the lemma), it suffices by Equation~\eqref{eq 2.2}  to show that: 

\begin{enumerate} \item[$(\star)$] If~$(r\text{ }r')$ is a transposition in the symmetric group on the roots~$f^{\circ k}$ and~$f(r) = f(r'),$ then~$(r\text{ }r')$ is realized by an element of the Galois group~$\Gal(F_{k}/F_{k-1}).$  \end{enumerate} We will say a transposition~$(r\text{ }r')$ on the set of roots of~$f^{\circ k}$ lies \emph{above} a root~$\pi$ of~$f^{\circ (k-1)}$ if~$$f(r) = f(r') = \pi.$$  We conclude the proof by demonstrating that~$(\star)$ holds.

First, we show that~$\Gal(F_{k}/F_{k-1})$ contains \emph{at least one} transposition above each root of~$f^{\circ (k-1)}.$ Fix a root~$\pi$ of~$f^{\circ (k-1)}.$ By Assumption~\ref{JC 3}, the automorphism~$\sigma_k \in \Gal(F_{k}/F_{k-1})$ acts on roots of~$f^{\circ k}$ as a transposition. Since~$\sigma_k$ is an element of~$\Gal(F_{k}/F_{k-1}),$ it necessarily lies above a root~$\pi'$ of~$f^{\circ (k-1)}.$ By Assumption~\ref{JC 1}, there is some~$\tau \in \overline{\langle{\sigma_0} \rangle}$ such that~$\tau(\pi') = \pi.$ The conjugate~$\sigma_k^{\tau}$ acts on the roots of~$f^{\circ k}$ as a transposition above~$\pi.$

To conclude the proof, we show that~$\Gal(F_{k}/F_{k-1})$ contains \emph{every} transposition above~$\pi.$ Observe that elements of~$\Gal(F_{k}/F_{k-1})$ which are~$\Gal(F_{k}/F_{k-1})$-conjugate to a transposition above~$\pi$ are also transpositions and lie above~$\pi$. We know~$\Gal(F_{k}/F_{k-1})$ contains some transposition above~$\pi$. To show~$\Gal(F_{k}/F_{k-1})$ contains all transpositions above~$\pi$, it suffices to show~$G_{F(\pi)}$ acts doubly transitively on~$X_{\pi}.$ 

Let~$F_{\pi}$ be the splitting field of~$f(X) - \pi$ over~$F(\pi).$ We want to show that~$G_{F(\pi)}$ acts doubly transitively on~$X_{\pi},$ we will show~$\Gal(F_{\pi}/F(\pi))$ is isomorphic to the symmetric group~$S_{X_\pi}.$ We use the following criterion for recognizing the symmetric group:

\begin{lemma}[pg. 98 \cite{gallagher1973large},  Lemma 4.4.3 \cite{serre2016topics}] \label{S_n} Let~$G$ be a transitive subgroup of~$S_n.$ Assume~$G$ contains a transposition. If~$G$ either contains 
\begin{enumerate}[label=(\roman*)]
\item\label{n-1 cycle Sn lemma} an~$(n-1)$-cycle,   or
\item\label{p cycle Sn lemma} a~$p$-cycle for some prime~$p >n/2,$ 
\end{enumerate}
then~$G = S_n$. 
\end{lemma} 

\noindent We show these conditions hold for~$\Gal(F_{\pi}/F(\pi)) \leqslant S_{X_\pi}.$ First, by Assumption~\ref{JC 1}, the automorphism~$\sigma_0$ acts on the roots of~$f^{\circ k}$ as an~$n^k$-cycle. It follows~$ \sigma_0^{n^{k-1}}$ is an element of~$G_{F(\pi)}$ which acts on~$X_{\pi}$ as an~$n$-cycle. Consequently,~$\Gal(F_{\pi}/F(\pi))$ acts transitively on~$X_{\pi}.$ 

Next, consider the element~$\sigma := \sigma_{\infty,N}^{(n-a)^{k-N-1}}.$ If~$\pi_2$ is a root of~$f^{\circ k-1}$ contained in~$T,$ then~$\sigma$ fixes~$\pi_1$ and cyclically permutes the~$(n-a)$-vertices of~$T$ which lie above~$\pi_1$. It follows that the image of~$\sigma$ in~$\Gal(F_{\pi_1}/F(\pi_1))$ is either a~$(n-1)$-cycle, or has an order divisible by a prime~{$p := n-a >n/2.$} Taking a further power of~$\sigma$ if necessary, we deduce that there is a root~$\pi_1$ of~$f^{\circ k}$ such that the image of the permutation representation of~$\Gal(F_{\pi_1}/F(\pi_1))$ on~$X_{\pi_1}$ contains either an {$(n-1)$-cycle} or a~$p$-cycle for some prime~$p>n/2.$ By Hypothesis \ref{JC 1}, there is some {element~$\tau \in \overline{\langle{\sigma_0} \rangle}$} which maps~$\pi_1$ to~$\pi.$ Under such an element~$\tau,$ the  set~$X_{\pi_1}$ is mapped to~$X_{\pi},$ and the actions of~$\Gal(F_{\pi'}/F(\pi'))$ and~$\Gal(F_{\pi}/F(\pi))$ are intertwined. In particular, the cycle types occurring in~$\Gal(F_{\pi_1}/F(\pi_1))$  are the same~$\Gal(F_{\pi}/F(\pi)).$ By Lemma~\ref{S_n}, we conclude~$\Gal(F_{\pi}/F(\pi))\cong S_{X_\pi}.$ \end{proof}  

\begin{remark} Hypothesis~\ref{JC 1} of Lemma~\ref{Jordan's Criterion} can be replaced by the weaker assumption that~$T_f$ is a regular,~$n$-branching tree and~$G_F$ acts transitively on the branches of~$T_f,$ i.e. that~$f^{\circ k}$ is irreducible for all~$k.$  We have chosen to state Lemma~\ref{Jordan's Criterion} in this form, as it better indicates our strategy for the proof of the main theorem of Section~\ref{sec 3}.  
\end{remark} 

\section{Almost Surjective Representations}\label{sec 3} 

Fix an integer~$n \geq 2$ and a field~$E \subset{\overline{\Q}}$ that is ramified outside of finitely many primes in~$\Z$.  In this section, we give explicit examples of polynomials of degree~$n$ whose arboreal~$G_{E}$-representation contains~$\Gamma(1).$ In fact, many of our examples have surjective arboreal Galois representation. 

Given a non-zero rational number~$\alpha,$ define~$\alpha^+ \in \Z_+$ and~$\alpha^- \in \Z$ to be the unique positive integer and integer, respectively, such that~$(\alpha^+, \alpha^-) = 1$ and~${\alpha = \frac{\alpha^+}{\alpha^-}.}$ Our main theorem in this section is: 

\begin{theorem}\label{long} Let~$E/\Q$ be an extension which is unramified outside finitely many primes of~$\Z.$ Choose~$a<n$ to satisfy:
\begin{enumerate} [label=(\text{a.\arabic*}),ref=a.\arabic*]
\item\label{a1}  if~$n \leq 6,$ then~$a = 1,$
\item\label{a2}  if~$n\equiv 7  \mod 8,$ then~$a = 1,$
\item\label{a3} otherwise,~$n-a$ is a prime and~$a < n/2.$  
\end{enumerate}

Assume~$A \in \Q$ satisfies: 
\begin{enumerate} [label=(\text{A.\arabic*}),ref=A.\arabic*]
 	\item\label{A1}  if~$p$ is a prime which ramifies in~$E,$ then~$p$-adic valuation~$v_{p}(A) > 0,~$
	\item\label{A2}   there is a prime~$p_0$ which is unramified in~$E$ and prime to~$n$ such that~$v_{p_0}(A) =1,$

\item\label{A3}~$A >2^{\frac{1}{n-1}} \left(\frac{a}{n}\right)^{-\frac{a}{n-1}}\left| \frac{a}{n} - 1\right|^{-\frac{n-a}{n-1}} > 1,$
\item\label{A4}~$v_2(A) \geq \frac{3}{n-1} + \frac{n}{n-1}v_2(n),$ 
\item\label{A5}~$(A^+,n) =2^{v_2(n)},$ 
\item\label{A6}~$(A^-,a(a-n)) = 1,$ 
\item\label{A699} there is a prime~$p_\infty >n$ which is unramified in~$E$ such that~$v_{p_\infty}(A) =-1,$ and 
\item\label{A7}    if~$n$ is even, then~${A^- \not\equiv \pm1 \mod 8,}$
\end{enumerate}
\noindent then the polynomial~$$f(X) := X^{a}(X-A)^{n-a} + A$$ and all of its iterates are irreducible over~$E$ and the image of the arboreal~$G_E$-representation associated to~$f:$
\begin{enumerate}
\item contains~$\Gamma(1)$ if~$a = 1$ and~$n>2,$  (i.e.~$n$ satisfies~$2 < n \leq 6$ or n~$\equiv -7 \mod 8$),   and 
\item equals~$\Aut(T_f),$ otherwise.  
\end{enumerate} 
\end{theorem}

\vspace{6pt} 

\noindent It is clear that there infinitely many values of~$A$ satisfying Hypotheses \eqref{A1} - \eqref{A7}. The fact that there is a value of~$a$ satisfying Hypotheses \eqref{a1} - \eqref{a3} is a consequence of Bertrand's postulate. 

The remainder of this section constitutes the proof of Theorem~\ref{long}. Fix elements~$a<n$ and~$A \in \Q$ which satisfy the hypotheses of this theorem, and let~${f(X) = X^a(X-A)^{n-a} + A.}$ Let~${N=1}$ if~$a=1$ and~$n>2;$ otherwise, let~$N=0$. As in Section~\ref{sec 2}, for each non-negative integer~$k,$ we denote the extension of~$E$ generated by all roots of~$f^{\circ k}$ by~$E_k \subseteq \Qbar.$ Finally, for each prime~$p \in \Z,$ fix for once and for all an embedding~$i_p:\overline{\Q} \hookrightarrow \overline{\Q}_p.$ The map~$i_p$ induces an inclusion on Galois groups~$G_{\Q_p} \hookrightarrow G_{\Q}.$ Throughout the remainder of this note, we will regard~$\overline{\Q}$ as a subfield of~$\overline{\Q}_p,$ and~$G_{\Q_p}$ as a subgroup of~$G_{\Q}$ via these maps. We denote the maximal unramified extension of~$\Q_p$ by~$\Q_p^{un}.$  

We will use Lemma~\ref{Jordan's Criterion} to show that the image of~$G_E$ under~${\rho_f: G_{\Q} \rightarrow \Aut(T_f)}$ contains~$\Gamma(N).$ To do so, we will show that~$G_E$ contains a set of elements~$\{\sigma_k: k\in \mathbf{N} \cup\{\infty\}\}$ that satisfy the hypotheses of Lemma~\ref{Jordan's Criterion}, where~$\sigma_{\infty}$ denotes~$\sigma_{\infty,N},$ an element satisfying Hypothesis~\ref{JC 2}. As described in the introduction, our strategy will be to find a set of prime integers~${\{p_k : k \in \mathbf{N} \cup \{\infty\}\}}$ that are unramified in~$E$ and have the property that the inertia subgroup~$I_{p_k} \leqslant G_{\Q_{p_k}} \leqslant G_E$ contains an element~$\sigma_k$ satisfying the relevant hypothesis of Lemma~\ref{Jordan's Criterion}. The primes~$p_0$ and~$p_\infty$ are those primes described in Theorem~\ref{long} that satisfy hypotheses~$\eqref{A2}$ and~$\eqref{A699}$, respectively. The local behavior of~$\rho_f$ at these primes mimic the local behavior at~$0$ and~$\infty$ in the arboreal Galois representation attached to~$f(X,t) = X^a(X-t)^{n-a} + t$ over~$\C(t).$ In Lemmas~\ref{eisenstein} and~\ref{ split infinity}, we show that when~$k$ is~$0 \text{ or } \infty,$ the~$I_{p_k}$-action on~$T_f$ factors through its tame quotient, and a lift~$\sigma_k$ of any generator of tame inertia satisfies the relevant hypothesis of Lemma~\ref{Jordan's Criterion}. From Lemma~\ref{eisenstein}, we will also deduce all iterates of~$f$ are separable. The primes~$p_k$ for~$k$ a positive integer are found in Lemma~\ref{new prime}. Every iterate of the polynomial~$f$ has a critical point at~$\frac{a}{n}A.$ Therefore,~$f^{\circ k}(\frac{a}{n}A)$ divides the discriminant of~$f^{\circ k}.$ Furthermore,~$\frac{a}{n}A$ is a simple critical point of~$f.$  In Lemma~\ref{new prime}, we find a prime~$p_k$ that is prime to the numerator of~$A$ (and hence by Assumption \eqref{A1} is unramified in~$E$) and divides the numerator of~$f^{\circ k}(\frac{a}{n}A)$ to odd order. Assumptions~$\eqref{A3} - \eqref{A6}$ and~$\eqref{A7}$ are made to guarantee that such a prime divisor occurs. In Lemma~\ref{local transposition}, we show the ring of integers of~$E_k$ is simply branched over~$\Spec(\Z)$ at~$p_k.$ At such primes~$p_k,$ the elements of the inertia group~$I_{p_k}$ that act non-trivially on the roots~$f^{\circ k}$ act as a transposition~$\sigma_k$.   

\vspace{10 pt}

We begin by verifying that all iterates of~$f$ are separable and that Hypothesis \ref{JC 1} of Lemma~\ref{Jordan's Criterion} holds for~$f.$ Let~$p_0$ be a prime that satisfies Assumption \eqref{A2}. We wish to show that \emph{all iterates of~$f$ are separable, and that there is an element~$\sigma_0 \in G_E,$ which acts transitively on the branches of~$T_f.$} We will show that all iterates of~$f$ are separable over~$\Q_{p_0},$ and that there is an element~$\sigma_0 \in I_{p_0}$ which acts transitively on the branches. This is immediate consequence of the following lemma:

\begin{lemma}\label{eisenstein} Let~$a \in \Z_+$ and~$A\in \Q$ satisfy the assumptions of Theorem~$\ref{long}.$ Let~$p_0$ be a prime that witnesses Assumption \eqref{A2}. For all positive integers~$i,$ the polynomial~$f^{\circ i}$ is irreducible over~$\Q_{p_0}^{un}$ and splits over a cyclic extension.     \end{lemma}

\begin{proof} We show that~$f^{\circ i}$ is an Eisenstein polynomial over~$\Z_{p_0}$. By Assumption \eqref{A2}, the polynomial~$f$ has~{$p_0$-integral} coefficients, and satisfies the congruence~${f \equiv X^{n} \mod p_0.}$ Therefore,~$f^{\circ k} \in \Z_{p_0}[X]$ and satisfies the congruence~${f^{\circ k}(X) \equiv X^{n^k}  \mod p_0.}$ Noting that~$f(0) = A$ and that~$A$ is a fixed point of~$f,$ we conclude that~$f^k(0) = A,$ which is a uniformizer in~$\Z_{p_0}.$ Therefore,~${f^{\circ k} \in \Z_{p_0}[X]}$ is an Eisenstein polynomial. 

Since the degree~$\deg(f^{\circ i}) = n^{i}$ is prime to~$p_0,$ an Eisenstein polynomial of this degree is irreducible over~$\Q_{p_0}^{un}$ and splits over the cyclic, tame extension of~$\Q_{p_0}^{un}$ of ramification degree~$n^i$.
\end{proof}

Our next task is to verify that Hypothesis~\ref{JC 2} of Lemma~\ref{Jordan's Criterion} holds for~$f$. Note that the conditions~{\eqref{a1}-\eqref{a3}} of Theorem \ref{long} are those on~$a$ that appear in the statement of Lemma~\ref{long}. Therefore, we must show that \emph{there is a regular~${(n-a)}$-branching subtree~$T \subseteq T_f$ whose lowest vertex has height~$N,$ and an element~${\sigma_{\infty} \in G_E}$ which preserves~$T$ and acts transitively on the branches of T.} This claim is vacuously true if~$n = 2;$ in this case one can take~$T$ to be any branch of~$T_f$ and~$\sigma_{\infty}$ to be the identity. We may therefore restrict our attention to the case that~$n >2.$

Let~$p_{\infty}$ be a prime that witnesses Assumption~\eqref{A699} of Theorem~\ref{long}. Since~$p_{\infty} > n,$ the {pro-$p_{\infty}$-Sylow} of~$\Aut(T_f)$ is trivial and the action of~$I_{p_{\infty}}$ on~$T_f$ factors through its pro-cyclic, tame quotient. By the unramifiedness condition in~\eqref{A699}, we have~$I_{p_{\infty}} \leqslant G_E.$ To verify the Hypothesis~\ref{JC 2}, it thus suffices to show there is an~$I_{p_{\infty}}$-stable, regular,~$(n-a)$-branching tree~$T$ whose lowest vertex has height~$N$ such that~$I_{p_{\infty}}$-acts transitively on the branches of~$T.$ In Lemma~\ref{ split infinity}, we will find such a tree. 

Before proving Lemma~\ref{ split infinity}, we prove the following lemma, which explains the failure of our methods to produce surjective arboreal Galois representations in Theorem \ref{long} under the assumption that~$a=1.$ In Section \ref{sec 4}, we will utilize this lemma to produce examples of surjective arboreal Galois representations when~$n \equiv 7 \mod 8$ or~$n$ is in the range~$3\leq n \leq 6,$ i.e. in the cases that~$a=1.$

\begin{lemma}\label{infinity unramified}  Let~$l$ be a prime integer which does not divide~$n-1$. Assume that~$B\in \Q_{l}$ satisfies~$v_{l}(B) =-1.$
Then the polynomial$$g(X) := X(X-B)^{n-1} + B$$ splits completely over an unramified extension of~$\Q_{l}.$
\end{lemma} 
\begin{proof} Consider the polynomial~$$S(X) := B^{-1}f(B+X) = B^{-1} X^n + X^{n-1} + 1 \in \Z_{l}[X]$$ The polynomial~$S$ splits over a given field if and only if~$g$ does. We show~$S$ splits over an unramified extension of~$\Q_{l}.$ Consider the Newton polygon of~$S;$ it has one segment of slope~$0$ and length~$n-1,$ and one segment of length 1 and slope 1. It follows that~$S$ has~$n-1$ roots of valuation~$0$ and one root of valuation~$-1.$ The root of valuation~$-1$ is necessarily~$\Q_{l}$-rational. As for the roots of valuation~$0,$ since~$$S(X) \equiv X^{n-1} +1 \mod l$$ is separable, these roots have distinct images in the residue field. By Hensel's lemma, we conclude~$S$ splits over an unramifed extension  of~$\Q_{l}.$
\end{proof} 

\begin{lemma}\label{ split infinity}  Assume~$n>2.$ Let~$a \in\Z_+$ and~$A\in \Q$ satisfy the assumptions of Theorem~\ref{long}. Let~$p_\infty$ be a prime that witnesses Assumption \eqref{A699}.  Then there is a subtree~$T \subseteq T_f$ whose lowest vertex has height~$N$ which is ~{$I_{p_{\infty}}$-stable}, regular, and ~$(n-a)$-branching  such that~$I_{p_{\infty}}$ acts transitively on the branches of~$T.$
\end{lemma}

\begin{proof} 

Consider the subtree of~$T^{\infty}_f \subseteq T_f$ consisting of~$0$ and the roots~$r \in \Qbar_{p_{\infty}}$ of~$f^{\circ i}$ such that the valuation~$v_{p_{\infty}}(f^{\circ j}(r)) = -1$ for all non-negative integers~$j < i.$ Since the action of {$G_{\Q_{p_{\infty}}}$} on~$\Qbar_{p_{\infty}}$ preserves the valuation, the tree~$T^{\infty}_f$ is {$G_{\Q_{p_{\infty}}}$-stable}. 

We claim that~$T^{\infty}_f$ is a regular,~$(n-a)$-branching tree.  To see this, observe that if~$\epsilon$ is any element of~$\Qbar_{p_{\infty}}$ of valuation less than or equal to~$-1.$ Then the Newton polygon of~$$f(X) - \epsilon = X^a(X-A)^{n-a} + (A - \epsilon) = (A - \epsilon) + \sum_{j= a}^n  {{n-a}\choose{n-j}} A^{n-j} X^j$$ has two segments: one has length~$n-a$ and slope~$-v_{p_{\infty}}(A) = 1,$ and the other has length~$a$ and slope~$$\frac{v_{p_{\infty}}(A^{n-a})- v_{p_{\infty}}(A-\epsilon)} {a} = \frac{a-n - v_{p_{\infty}}(A-\epsilon)}{a} \leq \frac{a-n +1}{a} \leq 2 - \frac{n}{a},$$ which is less than~$1.$ It follows that the pre-image of~$\epsilon$ under~$f$ contains exactly~$n-a$ elements of valuation~$-1.$ Specializing to the pre-image tree of~$0,$ we deduce that the tree~$T^{\infty}_f$ is regular and~{$(n-a)$-branching}.

When~$a =1,$ by Lemma~\ref{infinity unramified}, the polynomial~$f$ splits completely over an unramified extension of~$\Q_{p_{\infty}}.$  In this case, choose~$T$ to be any of the~$(n-a)$ full subtrees of~$T^{\infty}_f$ whose lowest vertex has height~$1.$ The inertia group~$I_{p_{\infty}}$ acts on~$T.$ If~$a > 1,$ let~$T$ equal~$T^{\infty}_f.$ We claim that the inertia group~$I_{p_{\infty}}$ acts transitively on the branches of~$T.$

Let~$r_k$ be a root of~$f^{\circ k}$ contained in~$T^{\infty}_f.$ The ramification index of~$\Q_{p_{\infty}}(r_k)/\Q_{p_{\infty}}$ is the size of the orbit of~$r_k$ in~$\Qbar_{p_{\infty}}$   under~$I_{p_{\infty}}.$ We wish to show that~$I_{p_{\infty}}$ acts transitively on~$T.$ By induction on~$k,$ it suffices to show that~$r_k$ orbit has size:
\begin{equation} e_k := \begin{cases} (n-a)^k, &\mbox{if }  a >1, \text{ and } \\ (n-a)^{k-1}, &\mbox{if } a=1. 
\end{cases} \end{equation} 
We show~$e(\Q_{p_{\infty}}(r_k)/\Q_{p_{\infty}}) = e_k.$ Note that~$e(\Q_{p_{\infty}}(r_k)/\Q_{p_{\infty}})$ is at most~$e_k$ as the size of the orbit of~$r_k$ under~$I_{p_{\infty}}$ is at most the number of vertices in~$T$ that have height~$k$ in~$T^{\infty}_f.$ To conclude the of proof, it suffices to show that~$e_k$ greater than or equal to~$e(\Q_{p_{\infty}}(r_k)/\Q_{p_{\infty}}).$ 

 We will show a root~$r_k$ of~$f^{\circ k}~$ contained in~$T^{\infty}_f$ satisfies:
\begin{equation}\label{valuation infty} v_{p_{\infty}}((r_{k} - A)) = 1 + \sum_{i=1}^k \frac{n-1}{(n-a)^i}.\end{equation} 
For each integer~$i$ in the range~$0\leq i \leq k$ define~$$r_i := f^{\circ k-i}(r_k) \text{ and }  \epsilon_i := (r_i - A)/A.$$ Equation \eqref{valuation infty} is equivalent to the assertion that \begin{equation}\label{see-saw} v_{p_{\infty}}(\epsilon_0) = 0 \text{ and } v_{p_{\infty}}(\epsilon_i) =\frac{v_{p_{\infty}}(\epsilon_{i-1})}{n-a} +  \frac{n-1}{n-a}\text{ if } i>1.\end{equation} We verify \eqref{see-saw}. The case when~$i=0$ is clear, as~$\epsilon_0 = -1.$  Consider the case where~$i >0.$ Then since~$A(1 + \epsilon_i) = r_i,$ we see that~$\epsilon_i$ is a root of 
\begin{align*} g_i(X) &:= f(A(1+X)) - r_{i-1} \\&=  A^{n}(1+X)^aX^{n-a} + (A - r_{i-1}) \\&= A^{n}(1+X)^aX^{n-a} + \epsilon_{i-1}A. \end{align*} Examining the Newton polygon of~$g_i,$ one sees that~$g_i$ has exactly~$a$ roots of valuation~$0$ and~$n-a$ roots of valuation~$$-\frac{v_{p_{\infty}}(\epsilon_{i-1}A) - v_{p_{\infty}}(A^n)}{n-a}  =\frac{v_{p_{\infty}}(\epsilon_{i-1})}{n-a} +  \frac{n-1}{n-a} .$$ Since~$f -r_{i-1}$ has exactly~$n-a$ roots of valuation~$-1$, it must be the case that~$\epsilon_i$ is a root of~$g_i$ of valuation~$$\frac{v_{p_{\infty}}(\epsilon_{i-1})}{n-a} +  \frac{n-1}{n-a} > 0.$$ Hence, Equation \eqref{valuation infty} holds and~$e_k \geq e(\Q_{p_{\infty}}(r_k)/\Q_{p_{\infty}}).$ \end{proof}

We thus conclude that Hypothesis \ref{JC 2} of Lemma~\ref{Jordan's Criterion} holds for~$f.$

\vspace{10pt}

The final hypothesis of Lemma~\ref{Jordan's Criterion} is \emph{that for every positive integer~$k> N$ the permutation representation of~$\Gal(E_k/E_{k-1})$ acting on the roots of~$f^{\circ k}$ in~$E_k$ contains a transposition.} It is shown to hold for~$f$ for all values of~$k \geq 0$ by the following two lemmas. Recall our convention for writing a rational number as a fraction: for~$\alpha \in \Q,$ we denote by~$\alpha^+ \in \Z_+$ and~$\alpha^- \in \Z$  the unique \emph{positive} integer and integer, respectively, such that~$(\alpha^+, \alpha^-) = 1$ and~${\alpha = \frac{\alpha^+}{\alpha^-}.}$ 

Note that~$\frac{a}{n}A$ is a critical point of~$f,$ and therefore by the chain rule, a critical point of all iterates of~$f.$ The next lemma, Lemma \ref{new prime}, shows that for every~$k>0,$ there is a prime~$p_k$ (satisfying certain conditions), which does not divide~$A^+,$ so that ~$\frac{a}{n}A$ is a root of {$f^{\circ k} \mod p_k$}. By assumption \ref{A2}, all primes which ramify in~$E$ divide~$A^+.$ Hence,~$p_k$ is unramified in~$E.$ In Lemma \ref{local transposition}, we will show that under the Hypotheses of Lemma \ref{new prime} the inertia group~$I_{p_k}$ acts on the roots of~$f^{\circ k}$ as a transposition.      

\begin{lemma}\label{new prime}  Let~$a \in \Z_+$ and~$A\in \Q$ satisfy the assumptions of theorem~$\ref{long}.$ For each positive integer~$k,$ there exists a prime integer~$p_k \nmid nA^{-}A^+$ so that the~$p_k$-adic valuation of~$f^{\circ k}(\frac{a}{n}A)$ is positive and odd. 
\end{lemma}

\begin{proof}  

For each positive integer k, let~$c_k$ denote~$\frac{f^{\circ k}(\frac{a}{n}A)}{A}.$ To prove this lemma it suffices to show for all positive integers~$k$ that~$c_k^+$ is relatively prime to~$nA^{-}A^+$ and is not a perfect square. We will show the following. First, we show that~$c_k^+$ and~$A^+$ are relatively prime. Then, we show that~$c_k = c_k^+/c_k^-$ is a square in~$\Z_2^{\times}.$ To finish the proof, we analyze the denominator~$c_k^-.$ We show that if~$n_2 = n/2^{v_2(n)},$ then~$n_2A^-|c_k^-$ and that~$c_k^-$ is not a square in~$\Z_2^{\times}.$ Noting that~$2|A^+$ by Hypothesis~$\eqref{A4},$ these claims imply that~$nA^-A^+$ and~$c_k^+$ are relatively prime, and that~$c_k^+$ is not a square.   

Define~$c_0 = \frac{a}{n}.$ Then for all~$k>0,$ 
\begin{equation}\label{inductive formula} c_k = A^{n-1}c_{k-1}^a(c_{k-1} - 1)^{n-a} + 1. \end{equation}  Let~$p\neq 2$ be a prime integer factor of~$A^+.$ By Assumption~\eqref{A5}, the prime~$p$ is not a factor of~$n.$ Hence,~$c_0$ is~$p$-integral. Using Equation \eqref{inductive formula}, one concludes by induction that~$c_k$ is~$p$-integral and~$c_k \equiv 1 \mod p.$  

Now consider the case where~$p=2.$ By Hypothesis~\eqref{A4}, the valuation~$v_2(A)$ satisfies~$${v_2(A) \geq \frac{3}{n-1} + \frac{n}{n-1}v_2(n)> 0}.$$  Combining this with Equation \eqref{inductive formula}, we observe~$$v_2(c_1 - 1) = v_2\left(A^{n-1}\left(\frac{a}{n}\right)^a\left(\frac{a}{n} - 1\right)^{n-a} \right) \geq (n-1)v_2(A)  - nv_2(n) \geq 3,$$ and~$$v_2(c_k - 1) = v_2\left(A^{n-1}\left(c_{k-1} \right)^a\left(c_{k-1} - 1\right)^{n-a} \right) \geq v_2(c_{k-1} -1),$$ if~$k >1.$ Therefore,~$c_k$ is~$2$-integral and congruent to~$1 \mod 8.$ We conclude that~$c_k^+$ and~$A^+$ are relatively prime. Furthermore, recalling that the squares in~$\Z_2^{\times}$ are exactly the elements congruent to~$1 \mod 8,$ we conclude that~$c_k$ is a square in~$\Z_2^{\times}.$ 

Now, we examine~$c_k^-.$ We've seen that~$c_k^-$ is prime to 2.  Let~$n_2 := n/2^{v_2(n)}.$ We will show by induction that \begin{equation}\label{denominator} c_k^- = (A^-)^{n^k -1}n_{2}^{n^k}(-1)^{(n-a)n^{k-1}}.\end{equation} This equation shows that~$c_k^+$ is prime to~$n_2A^-.$  More subtly, Equation \eqref{denominator} shows~${c_k^- \not\equiv 1 \mod 8},$ and therefore is not a square in~$\Z_{2}^{\times}$. To see this, observe that
\begin{align*}  (A^-)^{n^k -1}n_{2}^{n^k}(-1)^{(n-a)n^{k-1}} & \equiv
\begin{cases} \pm A^- \mod 8  &\mbox{if } n \equiv 0 \mod 2 \\
(-1)^{n-a} \mod 8  & \mbox{if } n \equiv 1 \mod 8\\
 \pm n \mod 8  &\mbox{if } n \equiv 3,5 \mod 8 \\
  n(-1)^{(n-a)} \mod 8  &\mbox{if } n \equiv 7 \mod 8. 
  \end{cases} \\ & \equiv \begin{cases} \pm 3 \mod 8 &\mbox{if } n \equiv 0 \mod 2, \text{ by Assumption~\eqref{A7},}   \\
-1 \mod 8 & \mbox{if } n \equiv 1 \mod 8,  \text{ by Assumption~\eqref{a3},}  \\
 \pm 3 \mod 8 &\mbox{if } n \equiv 3,5 \mod 8 \\
  -1 \mod 8 & \text{if }
       \!\begin{aligned}[t]
       n \equiv 7 \mod 8, &\text{ as~$ n-a = n-1$ is even} \\ &\text{ by Assumption~\eqref{a2}.} \end{aligned}
  \end{cases} 
\end{align*}   
Hence, to conclude the proof, it suffices to confirm Equation~\eqref{denominator}. 

We will prove Equation~\eqref{denominator} by induction on~$k.$ We begin by showing the equation holds when~$k =1.$ The element~$$c_1 = A^{n-1}\left(\frac{a}{n}\right)^a\left(\frac{a}{n} - 1\right)^{n-a} + 1 = (-1)^{n-a}
\frac{(A^+)^{n-1}a^a(n-a)^{n-a}}{(A^-)^{n-1}n^n} +1.$$ So a prime~$p$ divides~$c_1^-$ only if~$p|A^-$ or~$p|n_2.$ To deduce Equation~\eqref{denominator} in this case, we must show that for all~$p|A^-n_2$ the valuation: \begin{equation}\label{valuation c_1} v_p(c_1^-) = v_p((A^-)^{n-1}n_2^n), \end{equation}  and the sign \begin{equation}\label{sign} \frac{c_1^-}{|c_1^-|} = (-1)^{n-a}. \end{equation} These equalities hold if and only if

\begin{equation}\label{gcd c_1} (A^-n_2, A^+a(n-a)) = 1, \end{equation} and  \begin{equation}\label{ineq c_1} \frac{(A^+)^{n-1}a^a(n-a)^{n-a}}{(A^-)^{n-1}n^n} > 1, \end{equation} respectively. We prove~\eqref{gcd c_1} and~\eqref{ineq c_1}. By Assumption~\eqref{A6}, if~$p$ divides~$n_2,$ then~$p$ is prime to~$A^+$. Since~$a$ and~$n$ are relatively prime, a prime~$p$ dividing~$n_2$ does not divide~$a(n-a).$ Similarly, if~{$p$ divides~$A^-,$} then by definition~$p$ is prime to~$A^{+},$ and by Assumption~\eqref{A6}, the prime~$p$ does not divide~$a(n-a)$. We conclude Equation~\eqref{gcd c_1} holds.  To see~\eqref{ineq c_1}, observe that
\begin{equation}\label{bigger than 2}\frac{(A^+)^{n-1}a^a(n-a)^{n-a}}{(A^-)^{n-1}n^n} = (A\left(\frac{a}{n}\right)^{\frac{a}{n-1}}\left| \frac{a}{n} - 1\right|^{\frac{n-a}{n-1}})^{n-1} >2 \end{equation}
by Assumption~\eqref{A3}. We conclude Equation~\eqref{denominator} holds when~$k = 1.$

Now assume that Equation~\eqref{denominator} holds~$k \geq 1,$ we show Equation~\eqref{denominator} holds for~$k+1.$  Observe that
$$c_{k+1} = A^{n-1}c_{k}^a(c_{k} - 1)^{n-a} + 1 =
\frac{(A^+)^{n-1}(c_k^+)^a((c_k - 1)^+)^{n-a}}{(A^-)^{n-1}(c_k^-)^n} +1.$$ Hence, a prime~$p$ divides~$c_{k+1}^-$ only if~$p | A^-c_k^-.$ By induction, it follows that all prime divisors of~$c_{k+1}^-$ must divide~$A^-n_2.$ Note that, 
$$(A^-)^{n-1}(c_k^-)^n = (A^-)^{n-1}((A^-)^{n^k-1}n_2^{n^{k-1}})^n= (A^-)^{n^k -1}n_{2}^{n^k}.$$ Hence, to show Equation~\eqref{denominator}, it is sufficient to show for all~$p| A^{-}n_2$ the valuation \begin{equation}\label{valuation c_k} v_p(c_{k+1}^-) = v_p((A^-)^{n-1}(c_k^-)^n),\end{equation} and that the sign \begin{equation}\label{sign c_k} \frac{c_{k+1}^-}{|c_{k+1}^-|} = \left( \frac{c_{k}^-}{|c_{k}^-|} \right)^n.\end{equation} These equations are implied by
 \begin{equation}\label{gcd c_k} (A^-n_2, A^+c_k^+(c_k - 1)^+) =1, \end{equation} and 
\begin{equation}\label{size c_k} \left| \frac{(A^+)^{n-1}(c_k^+)^a((c_k - 1)^+)^{n-a}}{(A^-)^{n-1}(c_k^-)^n} \right| = \left| A^{n-1}c_{k}^a(c_{k} - 1)^{n-a} \right| = |c_{k+1} -1|  > 2   > 1,\end{equation} respectively.  

We conclude the proof by demonstrating equations~\ref{gcd c_k} and~\ref{size c_k}. Because~$n_2$ and~$A^+$ are relatively prime (by Assumption~\eqref{A5}), and~$A^-n_2$ divides~$c_k^-$ and~$A^-n_2$ divides~$(c_k-1)^-$ by induction, we conclude equality~\ref{gcd c_k} holds. By Equation~\eqref{bigger than 2}, we see that~$|c_{k} -1|  > 2$ when~$k = 1.$ It follows by induction that  
$$|c_{k+1} -1|  = \left| A^{n-1}c_{k}^a(c_{k} - 1)^{n-a} \right|  > \left| A|^{n-1} \right| \left| c_{k} \right|^a \left|(c_{k} - 1) \right|^{n-a} > 2^{n-a}.~$$ Hence, Equation~\eqref{size c_k} holds.\end{proof}

By Lemma \ref{new prime}, the prime~$p_k$ does not divide~$A^+.$ Therefore by Assumption~\eqref{A2}, this prime is unramified in~$E.$ To finish the proof of Theorem \ref{long}, we show that some element of the inertia group~$I_{p_k} \leqslant G_E$ acts on the roots of~$f^{\circ k}$ as a transposition.

\begin{lemma}\label{local transposition} Let~$a \in\Z_+$ and~$A\in \Q$ satisfy the assumptions of theorem~$~\ref{long}.$ Let~$p_k$ be a prime integer such that~$p_k\nmid n{A^-A+}$ and the~$p_k$-adic valuation of~$f^{\circ k}(\frac{a}{n}A)$ positive and odd, then
 
 \begin{enumerate}
 
 \item\label{double root} there is a factorization of~$f^{\circ k}(X) \equiv  g(X)b(X) \mod p_k$ as where~$g(X)$ and~$b(X)$ are coprime,~$g(X)$ is a separable, and~$b(X) = (X - \frac{a A}{n})^2$, and
\item\label{loc trans} the inertia group~$I_{p_k} \leqslant G_{\Q_{p_k}} \leqslant G_E$ acts on the set of roots~$f^{\circ k}$ in~$\overline{\Q}_{p_k}$ as a transposition. 
\end{enumerate} 
\end{lemma}

\begin{proof}[Proof of Claim~\ref{double root}] We show that~$\frac{a}{n}A$ is the unique multiple root of~$f^{\circ k}$  and its multiplicity is 2. 

We begin by showing~$\frac{a}{n}A$ is a multiple root of~$f^{\circ k}.$ A polynomial over a field~$F$ has a multiple root at~$\alpha\in \overline{F}$ if and only if~$\alpha$ is both a root and a critical point. By assumption, the value~$\frac{a}{n}A$ is a root of~$f^{\circ k} \mod p_k.$  To see~$\frac{a}{n}A$ is a multiple root, observe that 
\begin{equation} (f^{\circ k})'(X) = f'(X)\prod_{0 < i <k} f'(f^{\circ i}(X)) \end{equation} 
and 
\begin{equation}\label{cp f}\begin{aligned} f'(X) &= aX^{a-1}(X-A) + (n-a)X^a(X-A)^{n-a -1}\\&= X^{a-1}X^{n-a-1}(nX - aA), \end{aligned}\end{equation} and therefore~$\frac{a}{n}A$ is a critical point of~$f^{\circ k}.$  

Now assume~$c$ is a root of~$f^{\circ k} \mod p_k$ with multiplicity~$m>1.$ Let~$\overline{\Z}_{p^k}$ be the ring of integers of~$\overline{\Q}_{p_k}$ and~$\mathfrak{m}$ be its maximal ideal. Because~$f^{\circ k}$ is separable, there exists exactly~$m$ roots~$r_1, \ldots, r_m \in \overline{\Z}_{p_k}$ of~$f^{\circ k}$ such that~$r_i \equiv c \mod \mathfrak{m}.$ Let~$L(c) := \{r_1, \ldots, r_m\}.$ To prove Claim~\ref{double root}, it suffices to show~$c$ equals~$\frac{a}{n}A$ and~$m= |L(c)|$ equals~$2.$ 

For each pair of pair of distinct roots~$r$ and~$r'$ lifting~$c,$ let~$l(r,r')$ be the smallest positive integer such that~${f^{\circ l(r,r')}(r) = f^{\circ l(r,r')}(r').}$ Considering~$r$ and~$r'$ as vertices of the tree~$T_f,$ the value~$l(r,r')$ is the distance to the most common recent ancestor between~$r$ and~$r'$. Let~$$N(c) := \max\{l(r,r'): r,r' \in L(c) \}.$$  

We claim that if~$N(c)$ equals~$1$, then~$c$ equals~$\frac{a}{n}A$ and~$m$ equals~$2.$ To see why, assume~$N(c)$ equals~$1.$ Then~$r_1,\ldots, r_m$ are all roots of the polynomial~$f(X) -f(r_1).$ Therefore,~$c$ is a critical point of~$f(X) \mod \mathfrak{m}.$ From Equation~\eqref{cp f}, one observes that the critical points of~$f(X)$ are~$0,A$ and~$\frac{a}{n}A.$ By assumption~$f^{\circ k}(c) \equiv 0 \mod \mathfrak{m}.$ On the other hand, since~$A$ is a fixed point of~$f$ and~$f(0) =A,$~$$f^{\circ k} (0) = f^{\circ k}(A) = A \not\equiv 0 \mod \mathfrak{m}.$$ Thus,~$c$ must equal~$\frac{a}{n}A.$ The critical point~$\frac{a}{n}A$ has multiplicity~$1$. Therefore,~$m = L(c) = 2.~$ 

To finish the proof the claim, we must show~$N(c) =1.$ Assume this is not the case, and let~$r$ and~$r'$ be a pair of lifts such that~$l := l(r,r') > 1.$ Then~$f^{\circ l-1}(r)$ and~$f^{\circ l -1}(r')$ are distinct roots of the polynomial~$$g_{r,r'}(X) := f(X) - f^{\circ l}(r) = f(X) - f^{\circ l}(r')$$ which reduce to~$f^{\circ l-1}(c)$ modulo~$\mathfrak{m}.$ It follows~$f^{\circ l-1}(c)$ is a root of~$g'_{r,r'}(X) = f'(X),$ and hence equals~$A$ or~$0$ or~$\frac{a}{n}A.$ Since~$f^{\circ k}(c) \equiv 0 \mod p_k$ and~$$f^{\circ k - l-1} (0) = f^{\circ k- l-1}(A) = A \not\equiv 0 \mod p_k,$$ it must be the case that~$f^{\circ l-1}(c)$ equals~$\frac{a}{n}A.$ But this implies, as~$0 \equiv  f^{\circ k}(\frac{a}{n}A) \mod p_k$ by assumption, that  \begin{align*} 0 &\equiv  f^{\circ k}(\frac{a}{n}A) \mod p_k \\ &\equiv  f^{\circ k}(f^{\circ l-1}(c)) \mod p_k  \\ &\equiv   f^{l-1}(f^{\circ k}(c))  \mod p_k \\ &\equiv  f^{l-1}(0) \mod p_k \\ &\equiv A \mod p_k,\end{align*} a contradiction. 
\end{proof}

\begin{proof}[Proof of Claim~\ref{loc trans}] The factorization~$b(x)g(x) = f(x),$ appearing in Claim~\ref{double root}, lifts by Hensel's Lemma to a factorization~$$B(X)G(X) = f(X)$$ in~$\Z_{p_k}[X],$  where~$B(X)$ and~$G(X)$ are monic polynomials such that~$$B \equiv b \mod p_k \text{ and  } G \equiv g \mod p_k.$$  As~$g$ is separable,~$G$ splits over an unramified extension of~$\Q_{p_k}.$ To show~$I_{p_k}$ acts a transposition, we show the splitting field of~$B$ is a ramified quadratic extension of~$\Q_{p_k}.$ 

Consider the quadratic polynomial~$B(X+\frac{a}{n}A) = X^2 +B'(\frac{a}{n}A)X + B(\frac{a}{n}A).$ As

$$B'(\frac{a}{n}A)G(\frac{a}{n}A) + B(\frac{a}{n}A)G'(\frac{a}{n}A) = f'(\frac{a}{n}A) = 0,$$ and~$$G(\frac{a}{n}A) \equiv g(\frac{a}{n}A) \not\equiv 0 \mod p_k,$$ we observe~$v_{p_k}(B'(\frac{a}{n}A)) \geq v_{p_k}(B(\frac{a}{n}A)).$ It follows that the Newton polygon~$B(X+\frac{a}{n}A)$ has a single segment of slope~$\frac{v_{p_k}(B(\frac{a}{n}A))}{2}$ and width 2.  As

$$v_{p_k}(B(\frac{a}{n}A)) = v_{p_k}(f(\frac{a}{n}A)) - v_{p_k}(G(\frac{a}{n}A)) = v_{p_k}(f(\frac{a}{n}A))$$ the slope is non-integral. We conclude~$B(X + \frac{a}{n}A)$ is irreducible and splits over a ramified (quadratic) extension. 
\end{proof}

\vspace{10 pt}

Having verified that the conditions of Lemma~\ref{Jordan's Criterion} hold for~$f,$ we conclude that Theorem~\ref{long} is true. 

\section{Bridging the Gap} \label{sec 4}
Having proven Theorem~\ref{long}, we observe that our main theorem, Theorem~\ref{small}, holds in polynomial degrees~$n$ satisfying~$n \not\equiv 7 \mod 8$ and~$n\geq 6,$ or~$n=2.$ In this section, we prove that Theorem~\ref{small} holds in all remaining cases. 

Assume that either~$n \equiv 7 \mod 8,$ or~$n$ is in the range~$3\leq n \leq 6.$ Define~$$f(X,t) := X(X-t)^{n-1} +t \in \Q[t,X].$$ By Theorem~\ref{long}, there are infinitely many values of~$A \in \Q$ such that the image of the arboreal Galois representation~$\rho_{f(X,A) }: G_E \rightarrow \Aut(T_{f(X,A) })$ associated to the specialization~$$f(X,A) = X(X-A)^{n-1} + A \in \Q[X]$$ contains~$\Gamma(1).$  To prove Theorem~\ref{small}, we will use the Hilbert Irreducibility Theorem to show that for some infinite subset of these values the splitting field of the specialization~$f(X,A)$ over~$E$ is an~{$S_n$-extension}. For our first step, we calculate the geometric Galois group of the~$1$-parameter family~$f(X,t).$

\begin{lemma}\label{function field}  Let~$F$ be a field of characteristic~$0.$ The splitting field of the polynomial~$f(X,t)$ over~$F(t)$ is an~{$S_n$-extension}. 
\end{lemma} 

\begin{proof}  Without loss of generality, we may assume~$F$ is the complex numbers~$\C.$ Let~$$g(X,t) = f(X-t,-t) = X^n -tX^{n-1} -t.$$ It suffices to show that the splitting field of~$g(X,t)$ over~$\C(t)$ is an~$S_n$-extension. Let~${\pi: C_0  \rightarrow \mathbf{P}^1}$ be the \'etale morphism whose fiber above a point~$t_0 \in \C$ is the set of isomorphisms~$$\phi_t: \{0, \ldots, n-1\} \xrightarrow{\sim} \{r \in \C: g(r,t_0) = 0 \}.$$ Let~$C$ be a smooth, proper curve containing~$C_0,$ and let~$\pi:C \rightarrow \mathbf{P}^1$ be the map extending~${\pi: C_0 \rightarrow \mathbf{P}^1.}$  The splitting field of~$g$ is an~$S_n$-extension if and only if~$C$ is connected. We show the latter. 

We will analyze the monodromy around the branch points of~$\pi:C \rightarrow \mathbf{P}^1$. The cover~$C$ is ramified above the roots of 
\begin{align*} \Delta g(X,t)  &= n^n\prod_{ \substack{c \in \overline{\C(t)}, \\
\frac{\partial g}{\partial t}(c,t) = 0}} g(c,t)^{m_c} \\
&= n^ng(0,t)^{n-2}g(\frac{n-1}{n}t,t) \\&= n^n(-t)^{n-2}\left(\left(-\frac{1}{n} t  \right)\left(\frac{n-1}{n}t\right)^{n-1} -t\right) \\
&= n^n(-t)^{n-1}\left(\left(\frac{1}{n}  \right)\left(\frac{n-1}{n} t \right)^{n-1} + 1\right) \end{align*}
where~$m_c$ is the multiplicity of the critical point~$c$. Hence,~$\pi:C \rightarrow \mathbf{P}^1$ is branched at~$0$ and~$$\alpha_k := Me^{\frac{(2k+1)\pi i}{(n-1)} },$$ where~$k \in \{0,\ldots, n-2\}$ and~$M$ is a positive real number which is independent of~$k.$ Each of the branch points~$\alpha_k$ is simple.  One may check (though it is not relevant to our proof) that~$\pi: C \rightarrow \mathbf{P}^1$ is unramified at~$\infty;$ for a proof, see Lemma~\ref{infinity unramified}. We let~$D := \{0, \alpha_0, \ldots, \alpha_{n-2} \}$ denote the branch locus.  

Since~$g(X,t) = X^n - tX^{n-1} - t$ is~$t$-Eisenstein, it splits over~$\C[[t^{1/n}]].$ Observing that~$$t^{-1}g(Xt^{1/n},t) \equiv X^n - 1 \mod  t^{1/n},$$ it follows that each of the roots~$r$ of~$g$ in~$\C[[t^{1/n}]]$ satisfy~$$r = e^{2\pi i k /n}t^{1/n}  \mod t^{2/n}$$ for some unique value of~$k \in \{0, \ldots n-1 \}.$ Let~$pt_{\alpha_0 \rightarrow 0 }$ be the set~$(0, |\alpha_0|)\alpha_0 \in \C,$ i.e. the image of the straight line path from~$0$ to~$\alpha_0.$ Let~$s: pt_{\alpha_0 \rightarrow 0 } \rightarrow C$ be the unique holomorphic section of~$\pi: C \rightarrow \mathbf{P}^1$ such that~$$\lim_{t \rightarrow 0^+} \frac{s(t)(k)}{|s(t)(k)|} = e^{\frac{2\pi i k}{n}}e^{\frac{\pi i}{(n-1)n}} .$$ We consider the monodromy representation~$\varphi: \pi_1(\mathbf{P}^1 \setminus D, pt_{\alpha_0 \rightarrow 0 }) \rightarrow S_{n}~$  which maps a path~$p$ in~${\mathbf{P}^1 \setminus D}$ with endpoints in~${pt_{\alpha_0 \rightarrow 0 }}$ to~$\hat{p}(1)^{-1} \circ \hat{p}(0)$ where~$\hat{p}$ is the unique lift of~$p$ satisfying~${\hat{p}(0) = s(p(0)).}$ To show~$C$ is connected, it suffices to show~$\varphi$ is surjective. Our strategy will be to show that the generators of the symmetric group~$(0\text{ }1\text{ }2\ldots n-1)$ and~$(0\text{ }1)$ are contained in the image of~$\varphi.$ 

Consider a counterclockwise circular path~$p_0$ around~$0$ with endpoints in~$pt_{\alpha_0 \rightarrow 0 }.$ Since~$0$ is the only branch point contained in the circle bounded by~$p_0,$ the image of~$p_0$ under~$\varphi$ is the cycle~$(0\text{ }1\text{ }2\ldots n-1).$ Let~$p_1$ be a path with endpoint in~$pt_{\alpha_0 \rightarrow 0 }$ which bounds a punctured disk in~$\mathbf{P}^1 \setminus D$ around~$\alpha_0.$ Since the branch point~$\alpha_0$ is simple, the image of~$p_1$ under~$\varphi$ is a transposition. We claim~$\varphi(p_1) = (0\text{ }1).$ 

Let~$S$ be the set of complex numbers~$z$ which satisfies \begin{equation*} \frac{\pi}{n(n-1)} \leq \mathrm{Arg}(z) \leq  \frac{2\pi}{n} + \frac{\pi}{n(n-1)}. \end{equation*} Note that~$\alpha_0 \in S.$ Furthermore, observe the boundary rays of~$S$ are the two tangent directions by which the~$0$-th and~$1$-st root of~$g(X,t_0)$ (in the labeling given by the section s) converge to~$0.$  To show~$\varphi(p_1) = (0\text{ }1)$, we will demonstrate that 
\begin{enumerate} \item[$(\star)$] for all~$t _0 \in pt_{\alpha_0 \rightarrow 0 }$ there exists a unique pair of roots of~$g(X,t_0)$ contained in S.  \end{enumerate}

\noindent From~$(\star)$, one concludes by uniqueness~$\varphi(p_1) = (0\text{ }1).$ 

Since~$\alpha_0$ is a simple branch point contained in~$S$, when~$t_0$ is sufficiently close to~$\alpha_0$ there are at \emph{least} two roots in~$S$. On the other hand, as~$t_0$ approaches~$0,$ there is a unique pair of roots whose tangent directions are contained in~$S$. Hence for~$t_0$ sufficiently close to~$0,$ there are at \emph{most} two roots contained in~$S$. To prove~$(\star)$ for all~$t_0 \in pt_{\alpha_0 \rightarrow 0 },$ we will show that there is no value~$t_0 \in pt_{\alpha_0 \rightarrow 0 }$ such that~$g(X,t_0)$ has a root~$r$ whose argument equals~$\frac{\pi}{n(n-1)}$ or~$\frac{2\pi}{n} + \frac{\pi}{n(n-1)},$ i.e. roots cannot leave or enter the sector~$S$ as one varies~$t_0$ along~$pt_{\alpha_0 \rightarrow 0 }.$    

Assume for the sake of contradiction that there is a value~$t _0 \in pt_{\alpha_0 \rightarrow 0 }$ and a root~$r$ of~$g(X,t_0)$ such that~$ \mathrm{Arg}(r) = \frac{\pi}{n(n-1)}$ or~$ \mathrm{Arg}(r) = \frac{\pi}{n(n-1)} + \frac{2\pi }{n}.$ Then since~$g(r,t_0) = 0,$ one observes that~$$r^n = t_0(r^{n-1} + 1).$$ And so, \begin{align*} 
 \begin{split}
 \frac{\pi}{n-1} & \equiv \mathrm{Arg}(r^n) \mod 2\pi  \\ &\equiv \mathrm{Arg}(t_0) + \mathrm{Arg}(r^{n-1} +1) \mod 2\pi \\ &\equiv \frac{\pi}{n-1} + \mathrm{Arg}(r^{n-1} +1)  \mod 2\pi.
 \end{split} 
 \end{align*} 
 From which it follows~$\mathrm{Arg}(r^{n-1} +1)\equiv 0  \mod 2\pi.$ Note however, 
~$$\mathrm{Arg}(r^{n-1}) \equiv \begin{cases}  \frac{\pi}{n} \mod 2 \pi  &\mbox{if } \mathrm{Arg}(r) = \frac{\pi}{n(n-1)}, \text{ and} \\\
  2\pi -\frac{\pi}{n} \mod 2 \pi, &\mbox{if } \mathrm{Arg}(r) = \frac{\pi}{n(n-1)}. \end{cases}$$ Therefore,~$r^{n-1}$ is not a real number. It follows~$r^{n-1} +1$ is not real, and therefore has non-zero argument,  a contradiction. We conclude that there is no value~$t _0 \in pt_{\alpha_0 \rightarrow 0 }$ such that~$g(X,t_0)$ has a root with argument~$\frac{\pi}{n(n-1)}$ or~$\frac{\pi}{n(n-1)} + \frac{2\pi }{n}.$ Therefore,~$\varphi(p_1) = (0\text{ }1)$ and~$C$ is connected.   \end{proof}

We deduce our main theorem, Theorem~\ref{small}, via a Hilbert irreducibility argument.   
\paragraph{Proof of Theorem 1.2} 
If~$n\not\equiv 7 \mod 8$ or in the range~$3\leq n\leq 6$, then the theorem is a consequence of Theorem~\ref{long}. 

Assume that~$n\equiv 7 \mod 8$ or~$3\leq n\leq 6.$ Without loss of generality, we may assume~$E$ is a Galois extension of~$\Q.$ Let~$D$ be the unique positive, square-free integer which is divisible by the primes which ramify in~$E$ and those that divide~$n(n-1).$  In particular, note that~$2$ divides~$D$. Let~$B = D/(D, n-1).$ Consider the polynomial~$${h(X,t) = f(X,B^{-1}(1+Dt))  \in \Q[t,X].}$$ By Lemma~\ref{function field}, the polynomial~$h(X,t)$ has Galois group~$S_n$ over~$\Q(B^{-1}(1+Dt)) = \Q(t).$ Therefore by the Hilbert Irreducibility Theorem, there exists infinitely many values~$t_0\in\Z$ such that the splitting field~$K_{t_0}$ of~$h(X,t_0) = f(X,B^{-1}(1+Dt_0))$ is an~$S_n$-extension of~$\Q.$ Fix such a value~$t_0.$ We claim that there is a finite set~$L$ of prime integers which satisfy the following two conditions. \begin{enumerate} 
\item\label{L1} If~$l\in L,$ then~$l \nmid D.$ 
\item\label{L2} The closed, normal subgroup\footnote{the subgroup~$S_L$ is simply the absolute Galois group of the maximal extension of~$\Q$ in which all primes in~$L$ are unramified.}~$S_L \leqslant G_{\Q}$ generated by the inertia groups~$I_{l}$ for~$l\in L$ acts on the roots~$f(X, B^{-1}(1+Dt_0))$ as the full symmetric group~$S_n.$ \end{enumerate} 

Since there are no everywhere unramfied extensions of~$\Q,$ the set of primes which ramify in~$K_{t_0}$ satisfy Condition~$\ref{L2}.$ We show this set satisfies Condition~\ref{L1}, i.e. that~$K_{t_0}$  is unramified at all primes dividing~$D.$ 

Recall that~$D = B(D, n-1)$. If~$l$ divides~$B,$ then~$l$ is prime to~$n-1$ and the valuation~{$v_l(B^{-1}(1+Dt_0)) = -1.$} It follows by Lemma~\ref{infinity unramified}, that the extension~$L_{t_0}$ is unramified at~$l.$ On the other hand, if~$l$ divides~$n-1,$ then~$f(X,B^{-1}(1+Dt_0))$ has~$l$-integral coefficients and the discriminant:~\begin{align*} \Delta(f(X,B^{-1}(1+Dt_0)) &= n^n \prod_{c \in \overline{\Q}\text{ }  : \text{ } h'(c,t_0)=0} f(c,B^{-1}(1+Dt_0))^{m_c}\\ & =  n^n (B^{-1}(1+Dt_0))^{n-1}\left((B^{-1}(1+Dt_0))^{n-1}\left(\frac{1}{n} -1\right) + 1\right)\\ & \equiv B^{1-n} \mod l,   \end{align*} \noindent is prime to~$l.$ Hence,~$K_{t_0}$ is unramified at~$l.$ We conclude that~$K_{t_0}$ is unramified at all primes dividing~$D.$

To conclude the proof of the Theorem, we  perturb~$B^{-1}(1+Dt_0)$ in~$\prod_{l \in L} \Q_{l}$ to produce values of~$A$ for which~$f(X,A)$ has a surjective arboreal~$G_{E}$-representation. Let~$X_0$ denote the set of roots of~$f(X,B^{-1}(1+Dt_0))$ in~$\Qbar.$ Note that since the splitting field of~$f(X,B^{-1}(1+Dt_0))$ over~$\Q$ is~$S_n$-extension, the polynomial~$f(X,B^{-1}(1+Dt_0))$ is separable over~$\Q_l$. Let~$$\delta_l := \min \{|r_1 -r_2|_l: f(r_1,B^{-1}(1+Dt_0))=f(r_2,B^{-1}(1+Dt_0)) = 0 \text{ and } r_1 \neq r_2\}$$ be the minimum distance between a distinct pair of roots. By Krasner's Lemma, there exists an open ball~$U_l \subseteq \Q_{l}$ centered at~$B^{-1}(1+Dt_0)$ such that if~$A_l \in U_l$ and~$r$ is a root of~$f(X,B^{-1}(1+Dt_0)),$ then there is a unique root~$r(A_l)$ of~$f(X,A_l)$ such that~$|r - r(A_l)|_l <  \delta_l.$  Since the action of~$I_l$ on~$\Qbar_l$ preserves distances, the map~$r \mapsto r(A_l)$ is~$G_{\Q_l}$-equivariant. Identifying the set of roots of~$f(X,A_l)$ and~$f(X,B^{-1}(1+Dt_0))$ via this map, we see that for all~$A_l \in U_l$ the image of~$I_l$ in the symmetric group~$S_{X_0}$ is locally constant.

The group~$S_L$ is the normal closure of the group generated by the subgroups~$I_l$ for~$l \in L.$ Let~$U_L :=\prod_{l \in L} U_{l}.$ Since the action of~$S_L$ on~$X_0$ surjects onto~$S_{X_0},$ for all~$A \in U_L \cap \Q$ the permutation representation of~$S_L$ on the roots of~$f(X,A)$ is surjective. Since~$E$ is Galois and unramified at the primes in~$L$, the group~$G_E \leqslant G_\Q$ is normal and contains~$S_L.$ It follows that for any~$A \in U_L \cap \Q$ the splitting field of~$f(X,A)$ over~$E$ is an~$S_n$-extension.

We conclude the proof by showing that there are infinitely many values~$A \in U_L \cap \Q$ such that the arboreal Galois representation attached to~$f_{1,A}(X) := f(X,A)$ contains~$\Gamma(1).$ By Theorem~\ref{long}, it suffices show that there are infinitely many~$A \in U_L \cap \Q$ satisfying Hypotheses~\eqref{A1} -~\eqref{A7}. Let~$p_0$ and~$p_\infty$ be any choice of distinct primes which are greater than~$n,$ unramified in~$E,$ and not contained in~$L.$  Then Hypotheses~\eqref{A1} -~\eqref{A699} are open local conditions on~$A$ at the finite set of places dividing~$Dp_0p_{\infty}$ and~$\infty.$ In particular, they are conditions at places distinct from those in~$L.$ Let~$U_{\Gamma(1)}$ denote the open subset of~$\R \times \prod_{p|Dp_0p_{\infty}} \Q_p$ consisting of values which satisfy Hypotheses~\eqref{A1} -~\eqref{A699} locally. Let~$S$ denote the set of places~$$S := \{| \cdot |_p: p\in L ,\text{ or } p= \infty,  \text{ or } p | Dp_0p_{\infty} \}.$$ By weak approximation there are infinitely many values~$A_0 \in (U_{\Gamma(1)} \times U_L )\cap \Q.$ Fix any such value. Since~$U_{\Gamma(1)} \times U_L$ is open, there exists a real number~$\epsilon >0$ such that if~$|1-w|_{p} < \epsilon$ at all places in~$S,$ then~$wA_0 \in U_{\Gamma(1)} \times U_L.$ Fix such an~$\epsilon >0.$ Let~$M$ be a positive integer such that~$|M|_p < \epsilon$ at all \emph{finite} places~$|\cdot|_p \in S.$ If~$x$ is any positive integer which is
\begin{enumerate} \item\label{final1} not divisible by the primes contained in~$S,$ and \item\label{final2}  sufficiently large: specifically~$M/x< \epsilon,$\end{enumerate} then
$A_x:= \frac{x+M}{x}A_0 \in U_{\Gamma(1)} \times U_L,$ and therefore satisfies hypotheses~\eqref{A1} -~\eqref{A699}.  For such a value~$x \in \Z_{+},$ if one additionally asks that
\begin{enumerate}[resume] \item\label{final3}~$(x,A^+_0)=1$ and~$x \not\equiv \pm (A_0^{-})^{-1} \mod 8,$ \end{enumerate} then~$A_x^-\equiv A_0^-x  \not\equiv \pm 1\mod 8,$  and hence~$A_x$ satisfies hypothesis~\eqref{A7}. There are infinitely many~$x \in \Z_{+}$ satisfying conditions~\ref{final1},~\ref{final2}, and~\ref{final3}. For every such value, the arboreal~$G_E$-representation associated to~$f(X,A_x)$ is surjective. 
\section*{Acknowledgements}
The author would like to thank Nicole Looper for explaining her arguments in~\cite{looper2016dynamical}, the University of Chicago for its hospitality, and Mathilde Gerbelli-Gauthier for reading a preliminary draft of this work.

\bibliographystyle{amsalpha}
\bibliography{OdoniBib}
\end{document}